\documentclass[a4paper,12pt]{amsart}

\usepackage{amsthm,amssymb,amsmath,amscd}
\usepackage[mathscr]{euscript}
\usepackage{textcomp}

\theoremstyle{definition}
\newtheorem{definition}{Definition}

\theoremstyle{plain}
\newtheorem{theorem}[definition]{Theorem}
\newtheorem{lemma}[definition]{Lemma}

\newtheorem*{lemma*}{Lemma}

\newtheorem*{conj*}{Conjecture}

\theoremstyle{remark}

\newcommand{\rtg}{\Gamma(n,p)}
\newcommand{\diag}{\mathcal{D}}
\newcommand{\pres}{\langle S\ |\ R \rangle}
\newcommand{\E}{\mathbb{E}}
\newcommand{\eps}{\epsilon}
\newcommand{\davkd}{\widehat{\mathcal{D}}}

\begin{document}

\title[Collapse of random triangular group]
{A sharp threshold for collapse\\ of the random triangular group}

\author{Sylwia Antoniuk}

\address{Adam Mickiewicz University,
Faculty of Mathematics and Computer Science
ul.~Umultowska 87,
61-614 Pozna\'n, Poland}

\email{\tt antoniuk@amu.edu.pl}

\author{Ehud Friedgut} %\footnote{2nd author supported in part by I.S.F. grant 0398246, and BSF grant 2010247.}

\address{Weizmann Institute of Science,
Faculty of Mathematics and Computer Science
234 Herzl St. 
Rehovot 7610001, Israel}

\email{\tt ehud.friedgut@weizmann.ac.il}

\author{Tomasz \L{u}czak}

\address{Adam Mickiewicz University,
Faculty of Mathematics and Computer Science
ul.~Umultowska 87,
61-614 Pozna\'n, Poland}

\email{\tt tomasz@amu.edu.pl}

\thanks{Sylwia Antoniuk is partially supported by NCN grant 
2013/09/N/ST1/04251, Ehud Friedgut is supported in part by I.S.F. grant 0398246, and BSF grant 2010247, 
and Tomasz {\L}uczak is partially 
supported by NCN grant 2012/06/A/ST1/00261.}

\keywords {random group, collapse, sharp threshold, torsion-free}

\subjclass[2010]{
Primary:
  20P05; %Probabilistic methods in group theory
secondary:
	05C80, % random graphs
	 20F05. %Generators, relations, and presentations
 }

\date{March 13, 2014}

\begin{abstract}
The random triangular group $\rtg$ is the group given by a random 
group presentation with $n$ generators in which every relator 
of length three is present independently with probability $p$. 
We show that in the evolution of $\rtg$ the property 
of collapsing to the trivial group admits a very sharp threshold.
\end{abstract}

\maketitle

\section[Introduction]{Introduction}

Let $P = \pres$ denote a group presentation, where $S$ 
is the set of generators and $R$ is the set of relators. 
A group generated by a presentation $P$ is called a 
\textit{triangular group} if $R$ consists of cyclically reduced words of length three 
over the alphabet $S\cup S^{-1}$, that 
is if $R$ consists of words of the form $abc$ 
such that $a\neq b^{-1}$, $b\neq c^{-1}$ 
and $c\neq a^{-1}$. 
Here we consider the  
random triangular group $\rtg$ defined as a group given by a random triangular 
group presentation with $n$ generators and such that 
each cyclically reduced word of length three over the alphabet 
$S\cup S^{-1}$ is present in $R$ independently with probability $p=p(n)$.

We study the asymptotic properties of the random triangular group 
when the number of generators $n$ goes to infinity. Thus, for 
a group property $\mathcal{P}$ and a function $p(n)$, we say that 
$\Gamma(n,p(n))$ has $\mathcal{P}$ 
\textit{asymptotically almost surely} (a.a.s.), if the probability that 
$\Gamma(n,p(n))$ has this property tends to 1 as $n\rightarrow \infty$.

The notion of the random triangular group was introduced by \.Zuk \cite{Z2003}. 
In particular, he showed that 
%the property of collapsing to the trivial group 
%admits the following  threshold behaviour, i.e. 
for every   constant $\epsilon>0$, if $p \leq n^{-3/2-\epsilon}$, 
then a.a.s. 
$\rtg$ is an infinite, hyperbolic group, while for $p \geq n^{-3/2+\epsilon}$, 
a.a.s. $\rtg$ collapses to the trivial group (his result is stated for 
a somewhat different, yet equivalent, model of random triangular group).
 Antoniuk, \L uczak and \'{S}wi\c{a}tkowski~\cite{ALS2013} improved this result 
 from one side and showed that  there exists a constant 
$C > 0$ such that for $p \geq Cn^{-3/2}$ a.a.s.\ $\rtg$ collapses to the 
trivial group. They also asked if there exists a constant $c>0$  such that 
for $p<c n^{-3/2}$ a.a.s.\ $\rtg$ is infinite. 

 Note that  the property that a group is trivial is monotone, 
 i.e. if $\pres$ is trivial then for 
any $R'\supseteq R$ the group $\langle S\ |\ R' \rangle$ is trivial as well. 
Hence, by a well known argument of
Bollob\'as and Thomason~\cite{BT}, there exists a `coarse' 
threshold function for collapsibility i.e. 
there exists a function $\theta(n)$ such that if $p(n)/\theta(n)\to 0$, 
then a.a.s. $\rtg$ is non-trivial, whereas for 
$p(n)/\theta(n)\to\infty $ a.a.s. $\rtg$ collapses to the trivial group. 
However, the result  of Antoniuk, \L uczak and \'{S}wi\c{a}tkowski~\cite{ALS2013}
and their conjecture we have just mentioned suggest that $\rtg$ collapses  more rapidly, 
i.e. that the collapsibility has a `sharp' threshold.
Our  main result states  that this is indeed the case.

\begin{theorem}\label{th2}
Let $h(n,p)$ denote the probability that $\rtg$ is trivial. 
There exists a function $\tilde c(n)$ such that for any $\epsilon > 0$, 
$$\lim_{n\rightarrow\infty} h(n,(1-\epsilon)\tilde c(n)n^{-3/2}) = 0 \text{ and } 
\lim_{n\rightarrow\infty} h(n,(1+\epsilon)\tilde c(n)n^{-3/2}) = 1.$$
\end{theorem}

Unfortunately, the argument we use does not give any information on the 
asymptotic behaviour of $\tilde c(n)$. Nonetheless we strengthen the  conjecture from~\cite{ALS2013}
and predict that  $\tilde c(n)$ tends to a limit. 

\begin{conj*} There exists a constant $c>0$ such that for every constant $\eps>0$
the following holds.
\begin{enumerate}
\item[(i)] If $p\le (c-\eps)n^{-3/2}$,  then a.a.s.\ $\rtg$ is infinite and hyperbolic.
\item[(ii)] If  $p\ge (c+\eps)n^{-3/2}$, then a.a.s.\ $\rtg$ is trivial.
\end{enumerate}
\end{conj*}
 
 As we have already remarked it was shown in~\cite{ALS2013} that $$\limsup \tilde c(n) <\infty\,.$$
Although we cannot verify the conjecture and prove that $$\liminf \tilde c(n)>0,$$ 
we show however that $\tilde c(n)$ cannot tend to 0 too quickly.  

\begin{theorem}\label{th1}
Let $\omega$ and $p$ be  functions of $n$ 
such that $\omega(n)\to\infty$ as $n\to \infty$, and 
$$p(n) = n^{-3/2 - \omega/\log^{1/3}n}.$$
%Let $\rtg$ be a group given by a random presentation with $n$ generators 
%and relators chosen independently with probability $p=p(n)$  
%among all cyclically reduced words of length three 
%(i.e. the random triangular group).
Then a.a.s. $\rtg$ is infinite, torsion-free, and hyperbolic.
\end{theorem}

The structure of the paper is the following. 
In the next section we prove Theorem~\ref{th2}.
The argument is based on a result of the second author~\cite{Fr99} which, up to our 
knowledge, has never been used to show 
that  properties of  random groups have sharp thresholds. It states,
roughly, that if a property does not admit a sharp threshold then 
it is `local', i.e. its probability can be significantly 
changed by a local modification of the random structure 
(see Lemma~\ref{lemma:threshold} below).
We show that it is not the case with the collapsibility. 
In particular, we show that adding to $R$ a few more 
specially selected relators affects the probability of collapsing 
less than a tiny increase of 
the probability $p$, which in turn corresponds to 
adding to $R$ a small number of random relators.
Hence, a local modification of the random structure 
cannot have large impact on the probability of the property in question.

Then we prove Theorem \ref{th1}. We follow closely 
the argument of Ollivier who 
in \cite{O2004} showed that the assertion holds for some function 
$p(n)=n^{-3/2+o(1)}$.  This result was initially stated by Gromov~\cite{G1993} 
however it seems that Ollivier was the first one who 
gave a complete proof of this statement. 
We basically rewrite 
Ollivier's argument (who, following \.Zuk,  used a slightly different model 
of the random triangular group) to replace $o(1)$ in the power by some explicit 
function.

\section[Theorem~\ref{th2}]{Proof of Theorem~\ref{th2}}

As mentioned in the introduction, the tool we use in order to 
prove the sharpness of the threshold, 
as expressed in Theorem~\ref{th2}, is a result from Friedgut~\cite{Fr99}. 
In \cite{Fr99} the  author gives a general necessary condition for a property to 
have a coarse threshold, namely that it can be well approximated by a local property. 
Although the main theorem in that paper refers to graphs, the proof extends to 
hypergraph-like settings where the number of isomorphism types of bounded size 
is bounded. This includes random hypergraphs, random SAT Boolean formulae, 
and also the model of random groups that we are addressing in the current paper. 
A different, but very similar tool that can be used here is Bourgain's theorem 
that appears in the appendix of \cite{Fr99}, which has a weaker conclusion, 
but does not assume the symmetry of the property in question, such as we have 
in our current problem. To make things simpler we will use the ''working-mathematicians-version" 
of these theorems, as described in 
Friedgut~\cite {Fr05}.
We present below the lemma we will use, stated in terms of the problem at hand, 
but first let us introduce some notation. For each value of $n$ we denote by 
$S$ the set $S_n$ of generators, and assume that $S_n \subset S_{n+1}$, so that 
any fixed relator is meaningful for all sufficiently large values of $n$.
Next, let $\Gamma(n,p)$ be given by a presentation $P = \pres$, where $R$ is random, 
and let $R^*$ be a set of relators. We use the notation 
\[ h(n,p| R^*) := \text{Pr}[ \langle S\ |\ R \cup R^* \rangle \text{ is trivial}]. \] 
We  will use this notation both for $R^*=R_{\text{fixed}}  =\{r_1,\ldots, r_k\}$, 
a fixed set of cyclically reduced relators of length three, and  for $R^*= R_{\epsilon p}$, 
a random set of relators chosen from $S_n$ with probability $\epsilon p$
(in which case the probability is over both the choice of $R$ and of $R^*$).
The following is an adaptation of theorems 2.2, 2.3, and 2.4 from \cite{Fr05} to the current setting.

\begin{lemma}\label{lemma:threshold}
Assume that 
there exists a function $p=p(n)$, 
and constants $0< \alpha , \epsilon< 1$, such that there exist infinitely many values of $n$ 
for which it holds that 
\[
 \alpha < h(n,p)  < h(n,(1+\epsilon)p) < 1-\alpha.
\]
Then there exists a fixed (finite, independent of $n$ but possibly dependent on $\alpha$ 
and $\epsilon$) set $R_{\text{fixed}}  
=\{r_1,\ldots, r_k\}$ of cyclically reduced relators of length three, and a constant 
$\delta>0$ such that for all such $n$
\begin{enumerate}
\item[(i)]
$  h(n,p | R_{\text{fixed}}) > h(n,p) + 2 \delta$ 
\item[(ii)]
$h(n,p| R_{\epsilon p} ) < h(n,p) + \delta$
\end{enumerate}

\end{lemma} 

We will now see how this lemma, together with the fact that $\rtg$ collapses when 
$p=n^{-3/2+o(1)}$ (see either  Olliver~\cite{O2004}, or Antoniuk, \L uczak and \'{S}wi\c{a}tkowski~\cite{ALS2013} and Theorem~\ref{th1})
implies Theorem~\ref{th2}.

\begin{proof}[Proof of Theorem~\ref{th2}]
Assume, by way of contradiction, that Theorem~\ref{th2} does not hold. 
Then the assumptions of Lemma~\ref{lemma:threshold} are met.
Indeed, let $p_c = p_c(n)$ be defined so that $h(n,p_c) = 1/2$. 
Then specifically Theorem~\ref{th2} does not hold with the choice of $\tilde{c}(n) n^{-3/2} = p_c$. 
Hence there exists an $\epsilon_0 > 0$ and a positive constant $\alpha < 1/2$ 
such that for infinitely many values of $n$ 
either one has $h(n,(1 + \epsilon_0) p_c) < 1 - \alpha$, or  $h(n, (1 - \epsilon_0)p_ c) > \alpha$. 
In the first case $p = p_c$, $\alpha$, $\epsilon = \epsilon_0$ meet the assumptions of Lemma~\ref{lemma:threshold}, 
in the latter one can take $p = (1 - \epsilon_0) p_c$, $\alpha$, and $\epsilon = \epsilon_0 / (1 - \epsilon_0)$.

Now, let $R_{\text{fixed}} $ be the set of relators guaranteed by Lemma~\ref{lemma:threshold}, 
and let $Z := \{ z_1, z_2, \dots, z_\ell \}$ be the set of all generators involved in $R_{\text{fixed}}$ 
and all of their inverses.
Let $R_{\text{strong}}$ be the following relation: $z_1=z_2=\ldots = z_\ell=e$, where 
$e$ is the identity. Clearly
\[
 h(n,p | R_{\text{strong}}) \ge h(n,p | R_{\text{fixed}}).
\]
Next, consider the graph $G=(V,E)$, where 
$V = (S\cup S^{-1}) \setminus Z$, and $E$ consists of all pairs $xy$ such that there is 
a relator in $R$ which involves $x,y$ and an element of $Z$ (implying $x=y^{-1}$, 
since all elements in $Z$ are set by $ R_{\text{strong}}$ to be equal to the identity). 
The probability that a given pair $xy$ forms an edge is less than $6 \ell p$, 
and these events are independent, so $G$ can be coupled with the Erd\H{o}s-R\'{e}nyi
random graph $G(2n-\ell, q)$ with $q =O(n^{-3/2+o(1)})$.
Elementary first moment estimates imply that a.a.s.  $G$ has 
fewer than $n^{0.6}$ non-trivial  components each
of them consisting of at most two edges. Indeed, the expected number of connected subgraphs 
with exactly 3 edges is bounded by
\[(2n)^3 q^3 + 16(2n)^4 q^3 = O(n^{-0.5}),\]
hence by Markov's inequality a.a.s. there are no such subgraphs and isolated edges 
and paths of length two are the only non-trivial 
components. Moreover, the expected number of edges in $G$ can be bounded by
\[(2n)^2 q = O(n^{0.5}),\]
and again by Markov's inequality a.a.s. there are at most $n^{0.6}$ of them. 
Let $R'$ be the set of relators in $R$ that are disjoint from $Z$. 
%(Note that since $Z$ is finite a.a.s. there are no relators in $R$ containing two elements from $Z$.) 
Slightly abusing the notation we will also use $E$ to denote the set of relators 
$\{ xy : \{x,y\} \in E\}$. We have
\begin{equation} \label{eq:local}
 h(n,p | R_{\text{strong}}) = \text{Pr} [ \langle S \setminus Z\ |\ R' \cup E \rangle \text{ is trivial}] + o(n^{-0.4}) ,
\end{equation}
where the $o(n^{-0.4})$ accounts for the case where there exists in $R$ a relator involving two elements of $Z$.
Note that both $E$ and $R'$ are random.

Now let us consider the effect of $R_{\epsilon p}$. First 
let us choose arbitrarily a set $M$, $|M|=m=\lfloor n^{1.9}\rfloor$,  
of pairs of generators $\{a,b\}$, $a,b\in S\cup S^{-1}$.  
Define a graph $G' = (V',E')$  with $V' =  (S\cup S^{-1}) \setminus Z$ and $E'$ 
consisting of all pairs $xy$ such that $R_{\epsilon p}$ includes two relators 
of the form $abx$ and $aby^{-1}$, where $\{a,b\}\in M$.
Note that the existence of such two relators clearly implies that
$x=y^{-1}$. 
Let $X$ denote the  number of paths of length two in $G'=(V',E')$. 
It is easy to see that for the expectation of $X$ we have 
$$\E X\ge 
0.5 (2n)^3 m^2 (\eps p)^4=4 n^3 n^{3.8} n^{4(-3/2+o(1))}=4 
n^{0.8-o(1)}\ge 4n^{0.75}.$$
It is also easy to check that the standard deviation of $X$ is also 
of order $O(n^3m^2p^4)$, so from Chebyshev's inequality we infer 
that a.a.s. the number of such paths is larger than $3n^{0.75}$. 
On the other hand, let $Y$ be the number of 
pairs of paths which share at least one vertex.
The expectation 
of $Y$ is dominated by the number of pairs which share one edge
and is bounded from above by 
\[\E Y\le (2n)^4 m^3 (\eps p)^6 = 16n^4 n^{5.7} n^{6(-3/2+o(1))}=16n^{0.7+o(1)}.\]
Thus, from Markov's inequality, a.a.s. the number of such pairs is of order smaller 
than $n^{0.75}$. Consequently, a.a.s. $G' = (V',E')$  contains at least
$n^{0.75}\gg n^{0.6}$ disjoint paths of length two. 

Now  
\begin{equation} \label{eq:global}
 h(n,p | R_{\epsilon p}) = \text{Pr} [ \langle S \setminus Z \ |\ R' \cup E'  \rangle \text{ is trivial}] -o(n^{-0.4})\, ,
\end{equation}
where the term $o(n^{-0.4})$ accounts for the fact that even if the group generated by 
$  S \setminus Z $ collapses there are the generators in $Z$ to account for. 
However, if all generators in $  S \setminus Z$ are set to be equal to the identity 
it suffices that for each element $z \in Z$ there will be in $R$ a relator involving 
$z$ and two elements of $S \setminus Z$.
The probability of this event is at least as large as 
 $1-o(n^{-0.4})$ as $n$ tends to infinity.

We have shown  that~$G'$ contains at least $n^{0.1}$ edge-disjoint subgraphs isomorphic to  the graph spanned by the edges of $G$, i.e. there is a coupling which shows that $E'\supseteq E$.
Thus, the equations (\ref{eq:local}) and (\ref{eq:global}) contradict the items 1 and 2 in the conclusion of Lemma \ref{lemma:threshold}. 
Consequently,  the hypothesis of the lemma cannot hold, and the property in question must have a sharp threshold.
\end{proof}

\section[Theorem~\ref{th1}]{Proof of Theorem~\ref{th1}}

%We begin with the proof of the hyperbolicity of the random triangular group $\rtg$ 
%where $p=n^{-3/2-f(n)}$ and $f(n)$ is a function tending to 0 with $n\rightarrow\infty$ 
%to be determined later. 
In order to  show Theorem~\ref{th1} we need to introduce a number of 
somewhat technical definitions. Let $P = \pres$ be a group 
presentation. A~\textit{van Kampen diagram} with respect to the presentation $P$ 
is a finite planar 2-cell complex $\diag$ given with an embedding $\diag \subseteq \mathbb{R}^2$ 
and satisfying the following conditions.
\begin{itemize}
 \item $\diag$ is connected and simply connected,
 \item For each edge $e$ and one of its orientations we assign a generator $s\in S$. If we change 
 the orientation of $e$ to the opposite one, we replace the generator $s$ by $s^{-1}$.
 \item Each 2-cell $c$ is assigned a relator $r\in R$, the number of edges on the boundary of 
 $c$ is equal to the length of the relator $r$,
 %\item Each 2-cell $c$ has a marked vertex $v$ on its boundary and an edge 
 %containing $v$ orientation at this vertex,
 \item For each 2-cell $c$ there is a vertex $v$ such that the word read from $v$ 
in some direction of the boundary of the cell  is the 
 relator $r\in R$ assigned to $c$.
\end{itemize}

For a van Kampen diagram $\diag$ the \textit{size} of the diagram, denoted by 
$|\diag|$, is the number of faces (2-cells) of $\diag$. 
The \textit{boundary} of $\diag$, denoted by $\partial \diag$, is the boundary 
of the complement of $\diag$ in $\mathbb{R}^2$ and  
$|\partial \diag|$ denotes its size, that is the number of edges 
in $\partial \diag$. The \textit{boundary word} is any word read from 
some vertex in $\partial \diag$ in one of the directions around the boundary. In particular, 
the length of this word is precisely $|\partial \diag|$.

A~van Kampen diagram is said to be \textit{reduced} if there is no pair of 
cells $c$ and $c'$  sharing at least one edge $e$, which are assigned the same relator~$r$,
and are  such that if we read the word $r$ on the boundaries of $c$ and $c'$ 
the edge $e$ has the same orientation and corresponds to the same letter in the relator with 
respect to the starting point. A~van Kampen diagram is said to be \textit{minimal} 
if it is reduced and no other reduced van Kampen diagram with 
smaller number of faces has  the same boundary word.

Let $\Gamma$ be the group given by a presentation $P = \pres$. 
In order to verify whether $\Gamma$ is hyperbolic it is enough to 
consider minimal reduced van Kampen diagrams with respect to the presentation $P$ 
and to show that they fulfill a certain 
geometric condition. In particular, it is known that 
a group generated by 
a presentation $P = \pres$ is hyperbolic if and only if 
there exists a coefficient $\delta > 0$ such that
every minimal reduced 
van Kampen diagram $\diag$ with respect to the presentation $P$ 
satisfies the linear isoperimetric inequality 
$|\diag| \leq \delta|\partial\diag|$ (cf. \cite{A+1991}).

However, verifying that every reduced van Kampen diagram satisfies a certain 
isoperimetric inequality may turn out fairly hard since it requires 
showing that this inequality holds for all of them.
% and there are infinitely many of them.
At this point, the so called local to global 
principle for hyperbolic geometry 
(or Cartan-Hadamard-Gromov-Papasoglu theorem) (cf.\cite{P1996}) comes to 
an aid. This principle states that it is enough to verify  the 
isoperimetric inequality  for a finite family 
of van Kampen diagrams.  

\begin{theorem}[Cartan-Hadamard-Gromov-Papasoglu]\label{local-to-global}
Let $P = \pres$ be a triangular group presentation. Assume that for 
some integer $K > 0$ every minimal reduced van Kampen diagram $\diag$ w.r.t.~$P$ 
and of size $K^2/2 \leq |\diag| \leq 240K^2$ satisfies the inequality 
\[ |\diag| \leq \frac{K}{200}|\partial \diag|. \] 
Then for every minimal reduced van Kampen diagram w.r.t.~$P$ the following 
isoperimetric inequality is true 
\[|\diag| \leq K^2 |\partial\diag| .\]
\end{theorem}

Following Ollivier \cite{O2004}, in order to simplify the verification 
of the isoperimetric condition for van Kampen diagrams, we introduce
a $k$-labeled \textit{decorated abstract van Kampen diagram} (davKd). 
For simplicity, we do it only for groups with triangular presentations.
A $k$-labeled davKd is a finite planar 2-cell complex $\davkd$ 
given with an embedding $\davkd \subseteq \mathbb{R}^2$ 
and satisfying the following conditions:
\begin{itemize}
 \item $\davkd$ is connected and simply connected,
 \item each 2-cell $c$ is a triangle with a label $i$ from $\{1,\ldots,k\}$,
       with a marked vertex on its boundary and an orientation at this vertex,
 \item for all $i\in\{1,\ldots,k\}$ and for any 2-cell $c$ labeled by $i$, 
       starting from the marked vertex and going around according 
       to prescribed orientation at this vertex, the edges of $c$ get abstract labels $i_1, i_2, i_3$.
\end{itemize}

Now, let $P = \pres$ be a triangular presentation and 
consider a one-to-one map $\phi: \{1,\ldots,k\} \to R$ which assigns 
relators to faces of $\davkd$. Let $\phi_1(i), \phi_2(i), \phi_3(i)$ be the generators 
appearing on the first, second and third position of the relator $\phi(i)$ respectively. 
Then $\phi$ induces a map $\widetilde{\phi}:\{i_r\}_{1\leq i \leq k, 1 \leq r \leq 3} \to S$ 
which assigns to each abstract label $i_r$ a generator from $S$, namely 
$\widetilde{\phi}(i_r) = \phi_r(i)$. The map $\phi$ is called 
a \textit{fulfillment map}, if additionally whenever there is an edge with 
two abstract labels $i_r$, $j_s$, then $\widetilde{\phi}(i_r)=\widetilde{\phi}(j_s)$.
We say that a given davKd $\davkd$ is \textit{fulfillable} 
with respect to the presentation $P = \pres$ if there exists a fulfillment 
map $\phi: \{1,\ldots,k\} \to R$ for $\davkd$.

A~davKd is said to be \textit{reduced} if 
there is no pair of faces sharing at least one edge, which are assigned the same label $i$ 
and have opposite orientations, 
and such that the common edge gets the same abstract label $i_r$ from both faces. 
A~davKd is said to be \textit{minimal} if there is no other davKd with 
smaller number of faces and having the same boundary word (with respect to the abstract 
labels of edges).

Our aim is to show that for a function $f=f(n)= \omega/\log^{1/3}n$, where $\omega=\omega (n)\to\infty$  
and $p=n^{-3/2-f}$, a.a.s. all minimal reduced van Kampen diagrams $\diag$  
with respect to the random presentation in $\rtg$ satisfy the 
isoperimetric inequality with a coefficient $\delta = \delta(n)=
(200/f)^2$. But to do it, it is enough to verify this inequality 
for all minimal reduced $k$-labeled davKd's, so we show that the following 
statement holds.

\begin{lemma}\label{lem1}
Let $\omega=\omega(n)\to\infty$, $\omega < \log\log n$, 
$f=f(n)=\omega/\log^{1/3}n$, and $p=p(n)=n^{-3/2-f} $.
Then  a.a.s. for each minimal reduced $k$-labeled
davKd $\davkd$, fulfillable w.r.t. $\rtg$, we have 
\[|\davkd| 
\leq \big(200 /f\big)^2 |\partial\davkd| .\]
In particular, a.a.s. each minimal reduced van Kampen diagram $\diag$ w.r.t. 
$\rtg$ satisfies the linear isoperimetric inequality
\[|\diag| \leq \big(200 /f\big)^2  |\partial\diag|.\]
\end{lemma}

\begin{proof} Let  $f=f(n)= \omega/\log^{1/3}n$. 
From Theorem~\ref{local-to-global} it is enough to show that a.a.s.
each given davKd $\davkd$ of size at most $|\davkd|\le 240 (200/f)^2$
satisfies the linear isoperimetric inequality with the coefficient $1/f$. 
We do it in two steps. First, we show that for any
davKd $\davkd$ with size bounded by $|\davkd|\le 240 (200/f)^2$
one of the following two possibilities holds:
\begin{enumerate}
\item[(i)] $\davkd$ satisfies 
the isoperimetric inequality with the coefficient $1/f$; 
\item[(ii)] the probability 
that  $\davkd$ is fulfillable by $\rtg$  
is bounded from above by $n^{-f/2}$. 
\end{enumerate}
Using this dichotomy, we then show that the probability that 
there is a bounded size fulfillable davKd $\davkd$ not satisfying the isoperimetric inequality 
in question goes to 0 with $n\to\infty$. 
Hence, a.a.s. all sufficiently small fulfillable davKd's satisfy this inequality.

Let $\davkd$ be a davKd with $m = |\davkd|$ faces having $k$ distinct labels 
and with $l_1$ internal edges and $l_2=|\partial\davkd|$ boundary edges. 
If each face is assigned a different label, i.e. each cell of $\davkd$ corresponds to a different 
relator, the probability that $\davkd$ is fulfillable is bounded above by $n^{l_1+l_2}p^m$. 
This is in fact a rather easy case and showing that for all diagrams with 
different labels and fulfillable in $\rtg$ a.a.s. an isoperimetric inequality holds with a coefficient $1/f$ is 
rather straightforward. Indeed, assume that for a given $\davkd$ the isoperimetric 
inequality does not hold, that is $f m = f |\davkd| > |\partial\davkd| = l_2$. Notice also 
that $3m = 2l_1 + l_2 \geq 2l_1 + 1$ as there is at least one edge in the boundary 
of $\davkd$. Then the probability that $\davkd$ is fulfillable 
is bounded by $n^{l_1+l_2}p^m = n^{l_1+l_2}n^{(-3/2-f)m} \leq n^{-1/2}$. 
Moreover, as we will show later, the number of different davKd's $\davkd$ with $|\davkd|\le 240 (200/f)^2$
is of order much smaller than $n^{1/2}$, hence a.a.s. there are no such diagrams which are 
fulfillable and at the same time do not satisfy the isoperimetric inequality.

The main challenge is to deal with diagrams 
where some  of the labels may appear more than once. 
On one hand, this reduces the number 
of distinct relators used to fulfill the diagram. On the other hand, this 
also imposes some restrictions on the generators used in this assignment.
To control the influence of these two factors we follow an approach of Ollivier 
from~\cite{O2004}. To this end let 
 $m_i$ denote the number of faces labeled with $i$. 
Without loss of generality we may assume that $m_1 \geq m_2 \geq \ldots \geq m_k$.
We want to count the probability that $\davkd$ is fulfillable with respect to the  
random presentation given by $\rtg$. 
We introduce an auxiliary graph $G = G(\davkd)$ which  
captures all the constraints resulting from the structure of the davKd.
The vertices of the graph $G$ are the abstract labels $\{i_r\}_{1\leq i \leq k, 1\leq r \leq3 }$ 
and two vertices $i_r$, $j_s$ are adjacent if there is an edge in $\davkd$ 
carrying labels $i_r$ and $j_s$. We also define a family of induced subgraphs 
$G_1 \subset G_2 \subset \ldots \subset G_k$ of $G$, where $G_l$ is a subgraph of $G$ 
induced by vertices $\{i_r\}_{1\leq i \leq l, 1\leq r \leq3 }$. Let us remark that
the main reason why this approach to davKd is so convenient is the fact that 
$G$ contains no loops if and only if davKd is reduced.

Now, the number of connected components in the graph $G$ is the total number of distinct 
generators which can appear in relators used in the 
fulfillment map $\phi$ for $\davkd$. In some sense, this gives us  
the number of degrees of freedom we have while choosing relators for the fulfillment map. 
Indeed, if two vertices are adjacent in~$G$, 
then the corresponding abstract labels in $\davkd$ are mapped by 
$\widetilde{\phi}$ to the same generator. 
Therefore, if we denote the number of connected components 
in the graph $G$ by $C$, then we obtain the estimate
\[\text{Pr}(\davkd \text{ is fulfillable} )\leq n^{C}p^k.\]
A similar argument works for the graphs $G_l$,
which correspond to a partial assignment, namely we assign relators 
to faces bearing labels $1,\ldots, l$. 
Let $C_l$ denote the number of 
connected components in $G_l$. Then
\[\text{Pr}(\davkd \text{ is fulfillable} )\leq n^{C_l}p^l = n^{C_l - l(3/2 + f)},\]
therefore putting 
\[ d_l = C_l - l(3/2 + f) \]
we get the estimate
\[\text{Pr}(\davkd \text{ is fulfillable} )\leq n^{\min d_l}.\]
Thus, if for some $l$ we have $d_l < -f/2$, then 
\[\text{Pr}(\davkd \text{ is fulfillable} )\leq n^{-f/2}.\]
On the other hand, 
we claim that in the case of $\min d_l \geq -f/2$, the diagram $\davkd$ 
satisfies the isoperimetric inequality with the coefficient $1/f$.
Indeed, as was observed by  Ollivier \cite{O2004} (see p.613)  one gets that 
\[ |\partial \davkd| \geq 3|\davkd|(1-2d) + 2\sum_{l=1}^k d_l (m_l - m_{l+1}),  \]
where the parameter $d$ is the density of the random triangular 
group, which in our notation is equal to $1/2 - f/3$. Thus
\[ |\partial \davkd| \geq 2f|\davkd|+ 2\sum_{l=1}^k d_l (m_l - m_{l+1}). \]
Next, observe that $m_l - m_{l+1} \geq 0$ for every $l$ and $\sum m_l = |\davkd|$. Hence, 
if $\min d_l \geq -f/2$, then 
\[ |\partial \davkd| \geq 2f|\davkd| - f\sum_{l=1}^k (m_l - m_{l+1}) \geq f|\davkd|, \]
and we arrive at the desired isoperimetric inequality
\begin{equation}\label{eq:2}
|\davkd| \leq \frac 1f|\partial\davkd|.
\end{equation}

To complete our argument we use the local to global principle. 
In our case the coefficient $K$ from Theorem \ref{local-to-global} is equal to $200/ f$. 
We need to show that the probability that there exists a diagram of size 
at most $240 (200/f)^2$ violating the isoperimetric inequality (\ref{eq:2}) 
tends to 0.
If this is the case, then the random presentation in the 
$\rtg$ model a.a.s. meets the assumptions of the local to 
global principle, hence a.a.s. each diagram satisfies the isoperimetric inequality 
with the coefficient $(200/f)^2$.

First, we need to count the number of all possible davKd's with precisely $m$ faces. 
To do this we take the number of all possible triangulations 
of a polygon which consist of exactly $m$ triangles, and then for each triangle 
we choose the orientation in 2 ways, the starting point in 3 ways
and the label of this face in $m$ ways. 

A triangulation of a polygon with $m$ triangles has at most 
$m+2$ vertices. Thus, the number of such triangulations  
is bounded from above by the number of distinct triangulations $t(N)$ of a 
2-dimensional sphere with $N$ vertices, where $N \leq m+3$, which 
in turn we bound from above by $\alpha^m$ for some absolute 
constant $\alpha>0$ (see Tutte~\cite{T1962}). 
Hence,  the total number of davKd's with exactly $m$ faces can be bounded by
$\alpha^m \cdot 6^m \cdot m^m/m! \leq \beta^m$, where $\beta>0$ 
is an appropriate constant. 
Therefore, the probability that a fulfillable davKd of size at most 
$240 (200/f)^2$  
violates the isoperimetric inequality (\ref{eq:2}) is at most
\begin{equation*}\label{eq5}
\sum_{m\leq 240 (200/f)^2} \beta^m n^{-f/2} 
\leq \beta^{\frac{\gamma}{f^2}} n^{-f/2},
\end{equation*}
for some constant $\gamma > 0$. It is  easy to verify that the right hand side
of this inequality tends to 0 as $n\to\infty$ provided 
$f=f(n) =\omega / \log ^{1/3}n $. 
Hence, a.a.s. for every davKd $\davkd$ with
$|\davkd|\le 240 (200/f)^2$ fulfillable in $\rtg$ 
the isoperimetric inequality holds with a coefficient $1/f$ and 
so the assertion follows from Theorem~\ref{local-to-global}.
\end{proof}

\begin{proof}[Proof of Theorem~\ref{th1}]
As the group properties in question are monotone decreasing, it is enough to consider 
$p(n) = n^{-3/2 - \omega/\log^{1/3}n}$ with $\omega < \log\log n$.  
Observe first that a.a.s. $\rtg$  is aspherical, i.e. 
there exists no reduced 
spherical van Kampen diagram with respect 
to the random presentation $\rtg$.
Indeed, such  a spherical reduced van Kampen 
diagram  has zero boundary, 
which violates the isoperimetric inequality proved in Lemma~\ref{lem1}.
Since $\rtg$ is aspherical,  
it is torsion-free (see, for instance, Brown~\cite{Br}, p. 187).
Consequently,  a.a.s. $\rtg$ is an infinite, hyperbolic group.
\end{proof}

Let us conclude with a  remark that in order to show the conjecture we have to prove Theorem~\ref{th1} with  $\tau(n)=O(1/\log n)$ instead of 
$\tau(n)=\omega/(\log n)^{1/3}$. Such an improvement seems to require a new approach 
and, perhaps, a stronger version of Theorem~\ref{local-to-global}.

\bibliographystyle{plain}

\end{document}